\documentclass[12pt]{amsart}
\usepackage[margin=1.25in]{geometry}
\usepackage{latexsym}
\usepackage{amsmath}
\usepackage{amssymb}
\usepackage{amsthm}
\usepackage{stmaryrd}
\usepackage{amscd}
\usepackage{enumerate} 
\usepackage{amssymb} 
\usepackage{mathrsfs}
\usepackage[all]{xy}
\usepackage{color}
\usepackage{graphicx}
\usepackage{mathtools}
\usepackage{comment}
\usepackage{multirow}
\usepackage{hyperref}
\hypersetup{colorlinks=true}
\usepackage{tikz}
\usepackage{tikz-cd}
\usetikzlibrary{decorations.markings}
\tikzset{degil/.style={
            decoration={markings,
            mark= at position 0.5 with {
                  \node[transform shape] (tempnode) {$\backslash$};
                  }
              },
              postaction={decorate}
}
}

\makeatletter
  
  \@addtoreset{equation}{thm}
  \makeatother

\title{A cusp surface singularity is Frobenius liftable}
\author{Tatsuro Kawakami}
\address{Graduate School of Mathematical Sciences, University of Tokyo, 3-8-1 Komaba,
Meguro-ku, Tokyo 153-8914, Japan} 
\email{kawakami@ms.u-tokyo.ac.jp}
\author{Teppei Takamatsu}
\address{Department of Mathematics, Faculty of Science,
Saitama University,
255 Shimo-Okubo, Sakura-ku,
Saitama-shi, Saitama 338-8570,
Japan}
\email{teppeitakamatsu.math@gmail.com}

\def\phi{\varphi}
\def\epsilon{\varepsilon}
\def\tilde{\widetilde}

\def\log{\operatorname{log}}

\def\Spec{\operatorname{Spec}}

\def\Supp{\operatorname{Supp}}

\def\Ker{\operatorname{Ker}}
\def\Coker{\operatorname{Coker}}

\def\p{{\mathfrak p}}

\newcommand{\Q}{\mathbb{Q}}

\newcommand{\PP}{\mathbb{P}}

\newcommand{\sO}{\mathcal{O}}

\theoremstyle{plain}
\newtheorem{thm}{Theorem}[section] 

\newtheorem{prop}[thm]{Proposition}

\newtheorem{lem}[thm]{Lemma}
\theoremstyle{definition} 
\newtheorem{defn}[thm]{Definition}
\newtheorem{conv}[thm]{Convention}

\theoremstyle{remark}

\newtheorem{defn and notation}[thm]{Definition and Notation}
 
\newtheorem*{cl}{Claim}

\theoremstyle{plain}
\newtheorem{theo}{Theorem}

\newtheorem{COR}[theo]{Corollary}

\begin{document}
\tolerance = 9999

\begin{abstract}
We show that a cusp surface singularity is Frobenius liftable.
As a consequence, we determine all the Frobenius liftable surface singularities.
\end{abstract}

\maketitle
\markboth{Tatsuro Kawakami and Teppei Takamatsu}{A cusp surface singularity is Frobenius liftable}

%\setcounter{tocdepth}{1}
%\tableofcontents

\section{Introduction}

Throughout this note, we work over an algebraically closed field $k$ of
characteristic $p>0$.  Frobenius liftability ($F$-liftability, for short) asks for a lifting of a scheme
to $W_2(k)$, together with a lifting of its Frobenius.
Global $F$-liftability has been studied extensively
\cite{BTLM,AWZ,AWZ2}, while Zdanowicz investigated local
$F$-liftability and its relationship with classical $F$-singularities
\cite{Zda18}.  
Since the $F$-liftability of normal singularities is characterized by the splitting of reflexive version of the first Cartier operator, it is stronger than $F$-purity (see \cite[Section 3.3]{Kawakami-Takamatsu} for example).

In \cite{Kawakami-Takamatsu}, we proved that a plt surface singularity with reduced boundary
is $F$-liftable if and only if it is $F$-pure, with the single exception of
the boundary-free rational double point (RDP for short) of type $E_8^1$ in characteristic
$5$.  The remaining case in passing from plt singularities to arbitrary
$F$-pure surface singularities was the cusp case.  Recall that a cusp surface
singularity is a normal surface singularity whose reduced exceptional divisor
on the minimal resolution is either a cycle of smooth rational curves or an
irreducible rational curve with one node.  Cusp singularities are $F$-pure
\cite[Theorem~1.2]{Mehta-Srinivas(surface)}, but their $F$-liftability was left open \cite[Conjecture~4.11]{Kawakami-Takamatsu}.  Moreover,
\cite[Proposition~4.13]{Kawakami-Takamatsu} reduced the classification of $F$-liftable
surface singularities to this case.

The purpose of this note is to settle that problem.

\begin{theo}
\label{thm:main}
    Let $k$ be an algebraically closed field of characteristic $p>0$.
Let $(P\in X)$ be a cusp surface singularity over $k$.  Then
$(P\in X)$ is $F$-liftable.
\end{theo}

Combining Theorem~A with \cite[Theorem~A and Proposition~4.13]{Kawakami-Takamatsu}, we obtain the complete classification of Frobenius liftable surface singularities.

\begin{COR}
Let $k$ be an algebraically closed field of characteristic $p>0$.
Let $(P\in X,B)$ be a surface singularity over $k$ such that $B$ is
reduced.  Then $(P\in X,B)$ is $F$-liftable if and only if it is
$F$-pure and the following does not hold:
\[
    p=5,\qquad B=0,\qquad\text{and}\qquad
    (P\in X)\text{ is an RDP of type }E_8^1.
\]
\end{COR}

%%%%%%%%%%%%%%%%%%%%%%%%%%%%%%%%%%%%%%%%%%%%%%%%%%%%%%%%%%%%%%%%%%%%%%%%%%%%%%%%%%%%%%%%%%%%%%%%%%%%%%%%%%%%%%%%%%%%%%%%%%%%%%%%%%%%%%%%%%%%%%%%%%%%%%%%%%%%
\section{Preliminaries}

\subsection{Notation and terminology}\label{subsection:Notation and terminology}
Throughout the paper, we work over a fixed algebraically closed field $k$ of characteristic $p>0$ unless stated otherwise.
A \textit{variety} means an integral separated scheme of finite type. 
A \textit{surface} is a variety of dimension two.

\begin{defn}[F-liftability]\label{def:F-lift}
Let $X$ be a normal variety and $B=\sum_{r=1}^n B_r$ a reduced divisor on $X$, where every $B_r$ is an irreducible component.  
We denote the ring of Witt vectors of length two by $W_2(k)$.

We say that $(X,B)$ is \textit{$F$-liftable} 
if there exist 
\begin{itemize}
	\item a flat morphism $\widetilde{X} \to \Spec W_2(k)$ together with a closed immersion $i\colon X\hookrightarrow \widetilde{X}$, 
	 \item a closed subscheme $\widetilde{B}_r$ of $\widetilde{X}$ flat over $W_2(k)$ for all $r\in\{1,\ldots,n\}$, and 
	 \item a morphism $\widetilde{F}\colon \widetilde{X}\to \widetilde{X}$ 
     over the Frobenius of $W_2 (k)$
\end{itemize}
	such that 
\begin{itemize}
	\item the induced morphism $i\times_{W_2(k)}k \colon X\to \widetilde{X}\times_{W_2(k)} k$ is an isomorphism, 
	\item $(i\times_{W_2(k)}k)(B_r)= \widetilde{B}_r\times_{W_2(k)} k$ for all $r\in\{1,\ldots,n\}$, and 
    \item $\widetilde{F}\circ i=i \circ F$, and $\widetilde{F}^{*}(\widetilde{B}_r|_{\tilde{U}})=p(\widetilde{B}_r|_{\tilde{U}})$ for all $r\in\{1,\ldots,n\}$, where $\tilde{U}\subset \tilde{X}$ is the lift of the log smooth locus of $(X,B)$.
\end{itemize}
\end{defn}

\subsection{Cartier operator}

Let $Y$ be a smooth variety, and let $E$ be a reduced divisor on $Y$ with simple normal crossing (snc, for short) support.
Let $A$ be a $\Q$-divisor whose support of the fractional part $A-\lfloor A\rfloor$ is contained in $E$.
We define $\sO_Y$-modules $B\Omega_Y^i(\log E)(pA)$ and $Z\Omega_Y^i(\log E)(pA)$ to be the boundaries and the cycles of $F_*\Omega_Y^i(\log E)(pA)$ at $i$: 
\begin{align*} 
\label{eq:definition-of-B1} B \Omega_Y^i(\log E)(pA) &\coloneqq {\rm Im}\left(F_*\Omega_Y^{i-1}(\log E)(pA) \xrightarrow{F_*d}  F_*\Omega_Y^{i}(\log E)(pA)\right),\\ 
Z\Omega_Y^i(\log E)(pA) &\coloneqq \Ker\left(F_*\Omega_Y^{i}(\log E)(pA) \xrightarrow{F_*d}  F_*\Omega_Y^{i+1}(\log E)(pA)\right) \nonumber,
\end{align*}
where $F_*\Omega_Y^{i}(\log E)(pA)$ denotes $F_*\left(\Omega_Y^{i}(\log E)\otimes \sO_Y(\lfloor pA \rfloor)\right)$.
Then $B\Omega^i_Y(\log E)(pA)$ and $Z\Omega^i_Y(\log E)(pA)$ are locally free $\sO_Y$-modules \cite[Lemma 5.10]{KTTWYY1}.

By definition, we obtain the short exact sequence
\begin{equation}\label{eq:ZOmegaB}
    0\to Z\Omega_Y^i(\log E)(pA) \to F_{*}\Omega_Y^i(\log E)(pA) \to B\Omega_Y^{i+1}(\log E)(pA) \to 0.
\end{equation}
Taking $i=0$ and $A=-\frac{1}{p}E$ in \eqref{eq:ZOmegaB}, we obtain the short exact sequence
\begin{equation}\label{eq:OOB}
0\to \sO_Y(-E) \to F_{*}\sO_Y(-E) \to B\Omega^{1}_Y(\log E)(-E) \to 0.
\end{equation}

The Cartier isomorphism (\cite[Lemma 3.3]{Hara98}, \cite[equation (5.4.2)]{KTTWYY1}) gives the short exact sequence
\begin{equation} \label{eq:BZOmega'}
0 \to B\Omega^i_Y(\log E)(pA) \to Z\Omega^i_Y(\log E)(pA) \xrightarrow{C} \Omega^i_Y(\log E)(A) \to 0.
\end{equation}

Taking $i=1$ and $A=0$, we obtain the short exact sequence
\begin{equation} \label{eq:BZOmega}
0 \to B\Omega^1_Y(\log E) \to Z\Omega^1_Y(\log E)\xrightarrow{C} \Omega^1_Y(\log E) \to 0.
\end{equation}
We note that
\begin{equation} \label{eq:B=Blog}
B\Omega^1_Y(\log E)\coloneqq {\rm Im}\left(F_*\sO_Y\xrightarrow{F_*d}  F_*\Omega_Y^{1}(\log E)\right)={\rm Im}\left(F_*\sO_Y\xrightarrow{F_*d}  F_*\Omega_Y^{1}\right)=:B\Omega^1_Y
\end{equation}
holds.

We need the following theorem:
\begin{thm}[\textup{\cite[Variant 3.3.2]{AWZ} (see also \cite[Theorem 3.3]{Kawakami-Takamatsu})}]\label{thm:F-lift characterization}
    The pair $(Y,E)$ is $F$-liftable if and only if \eqref{eq:BZOmega} splits.
\end{thm}
% \begin{proof}
%     The assertion follows from 
% \end{proof}

\section{Proof of the main theorem}\label{sec:F-lift and Car operator}

We refer to \cite[Notation and terminology (6)]{Kawakami-Takamatsu} for the definition of a \emph{germ}.

\begin{conv}
    Let $(x\in X)$ be a germ of a cusp surface singularity, and let $\pi\colon Y\to X$ be the minimal log resolution with reduced divisor $E=\sum_i E_i$.
    Here, the minimal log resolution $\pi\colon Y \to X$ is constructed as follows. We first take the minimal resolution of $X$. If it is already a log resolution, then it coincides with the minimal log resolution. Otherwise, the exceptional divisor is an irreducible curve with a node, so the minimal resolution is not a log resolution. In this case, we obtain $\pi\colon Y\to X$ by blowing up the node of the exceptional divisor.
    In either case, note that
\[
K_Y+E=\pi^*K_X=0.
\]
holds. 
Thus, \[
\Omega_Y^1(\log E)^*\coloneqq \mathcal{H}\! \mathit{om}_Y(\Omega_Y^1(\log E),\sO_Y)=\mathcal{H}\! \mathit{om}_Y(\Omega_Y^1(\log E),\omega_Y(E))\cong \Omega^1_Y(\log E)\]
holds.
Note that $E_i\cong \PP^1$ for all $i$. 
\end{conv}

\begin{lem}\label{lem:trivial decomp}
    For every $i$, 
    $\Omega^1_Y(\log E)|_{E_i}=\sO_{E_i}^{\oplus 2}$ and
    $\big(F^{*}\Omega^1_Y(\log E)\big)\big|_{E_i}=\sO_{E_i}^{\oplus 2}$.
\end{lem}
\begin{proof}
    By the  restriction of the residue exact sequence (cf.~\cite[(1.4.2)]{Gra}), we have the following short exact sequence
    \[
    0 \to \Omega^1_{E_i}(\log (E-E_i)|_{E_i})\to \Omega^1_Y(\log E)|_{E_i} \to \sO_{E_i} \to 0.
    \]
    Note that $\Omega^1_{E_i}(\log (E-E_i)|_{E_i})=\omega_{E_i}(E-E_i)=\sO_{\PP^1}(-2+2)=\sO_{\PP^1}$
    Since we have
    \[\mathrm{Ext}^1(\sO_{E_i},\sO_{E_i})\cong H^1(\PP^1,\sO_{\PP^1})=0,\] we obtain
    \[
    \Omega^1_Y(\log E)|_{E_i}=\sO_{E_i}^{\oplus 2},
    \]
    and the first assertion holds.
    By pulling back by $F$, we obtain the second assertion.
\end{proof}

\begin{lem}\label{lem:vanishing}
    We have
    \[
    H^1(\Omega^1_Y(\log E)\otimes F_{*}\sO_Y(-E))=0.
    \]
\end{lem}
\begin{proof}
    By the formal function theorem \cite[Chapter III Theorem 11.1]{Har}, it suffices to show that the following claim:
    \begin{cl}\label{cln:vanishing}
        Take a cycle $Z=\sum_{i} r_iE_i$ satisfying $r_i \geq 0$ and $\sum_i r_i\geq 1$.
        Then \[
        H^1(Z, F^*\Omega^1_Y(\log E)\otimes\sO_Y(-E))=0.
        \]
        Here, $F^*\Omega^1_Y(\log E)\otimes\sO_Y(-E)$ denotes $(F^*\Omega^1_Y(\log E))\otimes\sO_Y(-E)$.
    \end{cl}
    \begin{proof}[Proof of the claim]
        We prove by the induction on $\sum_i r_i$.
        First, suppose that $\sum_i r_i=1$, i.e., $Z=E_i$ for some $i$.
        Then it follows from Lemma \ref{lem:trivial decomp} that
        \[
        F^{*}\Omega^1_Y(\log E)\otimes\sO_Y(-E)\otimes \sO_{E_i}=F^{*}(\Omega^1_Y(\log E)|_{E_i})\otimes \sO_{E_i}(-E)=
        \sO_{E_i}(-E)^{\oplus 2}.
        \]
        Since $n\coloneqq E_i\cdot (-E)=-E_i\cdot E_i - E_i\cdot (E-E_i)\geq 1-2=-1$, we have
        \[
        H^1(E_i,\sO_{E_i}(-E))=H^1(\PP^1, \sO_{\PP^1}(n))=0.
        \]

        Now, take a cycle $Z=\sum_{i} r_iE_i$ satisfying $\sum_i r_i\geq 1$.
        By negativity lemma, we can find $E_i$ such that $Z\cdot E_i<0$.
        Consider the short exact sequence
        \[
        0\to \sO_{E_i}(-Z+E_i) \to  \sO_Z \to \sO_{Z-E_i} \to 0.
        \]
        By induction hypothesis, it suffices to show that
        \[
        H^1(E_i, F^{*}\Omega^1_Y(\log E)\otimes \sO_Y(-E-Z+E_i))=0.
        \]
        By Lemma \ref{lem:trivial decomp}, we have
        \[
        F^{*}\Omega^1_Y(\log E)\otimes \sO_Y(-E-Z+E_i)\otimes \sO_{E_i}=\sO_{E_i}(-E-Z+E_i)^{\oplus 2}.
        \]
        Since $(E-E_i)\cdot E_i=2$, we have 
        \[
        \sO_{E_i}(-E-Z+E_i)=\sO_{\PP^1}(-2+m)
        \]
        where $m\coloneqq E_i\cdot (-Z)>0$.
        Thus, 
        \[
        H^1(\sO_{E_i}(-E-Z+E_i)^{\oplus 2})=H^1(\PP^1,\sO_{\PP^1}(-2+m))^{\oplus 2}=0,
        \]
        and the claim is proved.
\end{proof}
\end{proof}
% \begin{rem}
%     In the proof, the condition $E_i\cdot (E-E_i)\leq 2$ is essential.
% \end{rem}

The following lemma seems to be well-known but we include the proof for the convenience.
Note that we need the following lemma when $X$ is reducible.

\begin{lem}\label{lem:S_1-sheaf}
    Let $X$ be a Noetherian scheme and let $\phi\colon \mathcal{F}\to \mathcal{G}$ be a morphism of coherent $\sO_{X}$-modules.
    Suppose that $\mathcal{F}$ satisfies Serre's condition $(S_1)$ and $\phi$ is injective at every generic point of $X$.
    Moreover, suppose that $\Supp \mathcal{F} = X$.
    Then $\phi$ is injective.
\end{lem}
\begin{proof}
    Let $\mathcal{K}\coloneqq \mathrm{ker}(\phi)$.
    Take $\p\in \mathrm{Ass}(\mathcal{F})$.
    Then $\mathrm{depth}_{\p\sO_{X,\p}}\mathcal{F}=0$ as $\p\sO_{X,\p}\in \mathrm{Ass}(\mathcal{F}_{\p})$.
    Since $\mathcal{F}$ satisfies $(S_1)$, it follows that $\dim \sO_{X,\p} = \dim \Supp \mathcal{F}_{\p}=0$.
    By assumption, $\p\notin \Supp(\mathcal{K})$.
    Therefore, we have 
    \[
    \mathrm{Ass}(\mathcal{K})\subseteq \mathrm{Ass}(\mathcal{F})\cap \mathrm{Supp}(\mathcal{K})=\emptyset,
    \]
    and thus $\mathcal{K}=0$.
\end{proof}

\begin{prop}\label{prop:isom}
    The Frobenius pullback map
    \[
    H^1(E, \Omega^1_Y(\log E)\otimes \sO_E)\xrightarrow{F^*} H^1(E, \Omega^1_Y(\log E)\otimes F_*\sO_E)
    \] is an isomorphism.
\end{prop}
\begin{proof}
Set $\mathcal{B}_E \coloneqq \Coker (\sO_E \rightarrow F_* \sO_{E})$.
By the following commutative diagram with exact rows and columns:
    \begin{equation*}
\xymatrix{ 
&0\ar[d]& 0\ar[d]& 0\ar[d]& \\
0\ar[r]&\sO_Y(-E)\ar[r]\ar[d]& F_{*}\sO_Y(-E)\ar[r]\ar[d]& B\Omega^1_Y(\log E)(-E) \ar[r]\ar[d]& 0\\
           0\ar[r]& \sO_Y \ar[r]\ar[d] & F_*\sO_Y\ar[r]\ar[d]& B\Omega^1_Y \ar[r]\ar[d]& 0\\
           0\ar[r]& \sO_E\ar[r]\ar[d] & F_{*}\sO_E\ar[r]\ar[d] &  \mathcal{B}_E \ar[r]\ar[d]& 0\\
           &0& 0& 0& }
\end{equation*}
we obtain an isomorphism
\[
\mathcal{B}_E\cong B\Omega^1_Y/B\Omega^1_Y(\log E)(-E). 
\]
We refer to \eqref{eq:OOB} for the horizontal exact sequence.
Since both $B\Omega^1_Y(\log E)(-E)$ and $B\Omega^1_Y$ are locally free $\sO_Y$-module, they satisfy $(S_2)$.
By depth lemma \cite[\href{https://stacks.math.columbia.edu/tag/00LX}{Tag 00LX}]{stacks-project}, 
the cokernel $\mathcal{B}_E$ satisfies $(S_1)$.
Moreover, $\mathcal{B}_{E, \eta_i}$ has rank $p-1$ at the generic point $\eta_i$ of every irreducible component $E_i$. Hence we have $\Supp \mathcal{B}_E=E$.

Next, consider the following commutative diagram  with exact rows and columns:
    \begin{equation*}
\xymatrix{ 0\ar[r]&\sO_E\ar[r]\ar[d]& F_{*}\sO_E\ar[r]\ar[d]& \mathcal{B}_E \ar[r]\ar[d]& 0\\
           0\ar[r]& \bigoplus_i \sO_{E_i}\ar[r]\ar[d] & \bigoplus_i F_*\sO_{E_i}\ar[r]\ar[d]& \bigoplus_i B\Omega^1_{E_i} \ar[r]& 0\\
           & \mathcal{C}\ar[r]^{\cong} & F_{*}\mathcal{C},& & }
\end{equation*}
where $\mathcal{C}\coloneqq \Coker (\sO_E\to \bigoplus_i \sO_{E_i})$.
Since the right vertical map is generically an isomorphism and $\mathcal{B}_E$ satisfies $(S_1)$, it is injective on $E$ by Lemma \ref{lem:S_1-sheaf}.
Since $\dim\Supp(\mathcal{C})=0$, the sheaf $\mathcal{C}$ is a finite-dimensional vector space over $k$.
By the snake lemma, the
map $\mathcal{C}\to F_*\mathcal{C}$ is injective, hence an isomorphism.
By the snake lemma again, the right vertical map
\[
\mathcal{B}_E\cong \bigoplus_i B\Omega^1_{E_i}
\]
is an isomorphism.
Therefore, it suffices to show that
\[
H^j(E_i,B\Omega^1_{E_i}\otimes \Omega^1_Y(\log E))=0
\]
for all $i,j$.
By Lemma \ref{lem:trivial decomp}, we have
\[
H^j(E_i,B\Omega^1_{E_i}\otimes \Omega^1_Y(\log E))=H^j(E_i,B\Omega^1_{E_i})^{\oplus 2}=H^j(\PP^1,B\Omega^1_{\PP^1})^{\oplus 2}=0,
\]
and we conclude.
\end{proof}

\begin{thm}
    The pair $(Y,E)$ is $F$-liftable.
    In particular, $X$ is $F$-liftable.
\end{thm}
\begin{proof}
    By birational descent of $F$-liftability \cite[Proposition 3.9]{Kawakami-Takamatsu}, the second assertion follows from the first one.
    By Theorem \ref{thm:F-lift characterization} and \eqref{eq:B=Blog}, the pair
    $(Y,E)$ is $F$-liftable if and only if
    the short exact sequence
    \[
     0\to B\Omega^1_Y \to  Z\Omega^1_Y(\log E) \to  \Omega^1_Y(\log E) \to 0
    \]
    splits.
    Thus, it suffices to show that
    \[
    H^1(Y, \Omega^1_Y(\log E)\otimes B\Omega^1_Y)=\mathrm{Ext}^1(\Omega^1_Y(\log E), B\Omega^1_Y)
    =0,
    \]
    where we used $\Omega^1_Y(\log E)^*=\Omega^1_Y(\log E)$ for the first equality.
    Considering the short exact sequence
    \[
    0\to \Omega^1_Y(\log E)\to \Omega^1_Y(\log E)\otimes F_{*}\sO_Y \to \Omega^1_Y(\log E)\otimes B\Omega^1_Y \to 0.
    \]
    Since the dimension of all $\pi$-fibers is one, we have $H^2(\Omega^1_Y(\log E))=0$.
    Thus the desired vanishing can be reduced to showing
    the surjectivity of
    \begin{equation}\label{eq1}
        H^1(Y, \Omega^1_Y(\log E)) \to H^1(Y, \Omega^1_Y(\log E)\otimes F_{*}\sO_Y)
    \end{equation}
    Consider the following commutative diagram with exact rows and columns:
    \begin{equation*}
\xymatrix{
& H^1(Y, \Omega^1_Y(\log E)\otimes F_{*}\sO_Y(-E))\overset{\text{Lem}\,\ref{lem:vanishing}}{=}0\ar[d]\\
H^1(Y, \Omega^1_Y(\log E))\ar[r]^-{\eqref{eq1}}\ar[d]& H^1(Y, \Omega^1_Y(\log E)\otimes F_{*}\sO_Y)\ar[d]\\
            H^1(E, \Omega^1_Y(\log E)\otimes \sO_E)\ar[r]^-{\cong,\,\text{Prop}\,\ref{prop:isom}}\ar[d] & H^1(E,\Omega^1_Y(\log E)\otimes F_{*}\sO_E)\\
            H^2(Y, \Omega^1_Y(\log E)(-E))=0.&}
\end{equation*}
Here, we obtain $H^2(\Omega^1_Y(\log E)(-E))=0$ again by dimensional reason.
By diagram chasing, we conclude the surjectivity of \eqref{eq1}, and we conclude.
\end{proof}

{
\noindent\textbf{Declaration of generative AI.}
During the development of this work, the authors used ChatGPT, powered by GPT-5.6 Sol Pro. 
The outline of the proof of Theorem \ref{thm:main}
%and in particular the key idea and an initial version of the argument in Proposition 3.6, 
was suggested by the model.  
The authors subsequently reconstructed and revised the complete argument. The authors take full responsibility for all mathematical claims, proofs, and text in this note.
}

\section*{Acknowledgements}
The authors wish to express their gratitude to Fabio Bernasconi for valuable conversations.
Kawakami was supported by JSPS KAKENHI Grant number JP24K16897 and by Inamori Foundation.
Takamatsu was supported by JSPS KAKENHI Grant number JP25K17228.

\bibliography{hoge.bib}
\bibliographystyle{alpha}

\end{document}